\documentclass[12pt, a4paper]{article}

\usepackage{a4,amsmath,amssymb, amsthm, latexsym, color, graphicx,url}
\usepackage{tikz}
\usepackage{authblk}
\usepackage{fullpage}
\usepackage{enumitem}
\usepackage{xcolor}
\newtheorem{theorem}{Theorem}[section]
\newtheorem{proposition}[theorem]{Proposition}
\newtheorem{lemma}[theorem]{Lemma}
\newtheorem{corollary}[theorem]{Corollary}

\theoremstyle{definition}
\newtheorem{example}[theorem]{Example}
\newtheorem{definition}[theorem]{Definition}
\newtheorem{remark}[theorem]{Remark}

\newcommand{\E}{\overrightarrow{E}}

\title{Transfer of difference structures: a new semidirect product framework}

\author[1] {Sophie Huczynska}
\author[2] {Struan McCartney}
\author[3] {Carys Williams}
\affil[1,2,3]{School of Mathematics and Statistics, University of St Andrews, St Andrews, KY16 9SS, Scotland, UK; email:(SH) sh70@st-andrews.ac.uk, (SM) sm444@st-andrews.ac.uk}
\date{\small Keywords: difference sets, external difference families, near-factorizations, semidirect products}
\begin{document}
\maketitle

\begin{abstract}
Difference sets, external difference families and near-factorizations of groups are much-studied structures satisfying certain uniformity conditions on the differences or sums within or between elements of sets in groups.  Research has centred on abelian groups, though there are classic results (such as Dillon's Dihedral Trick for difference sets and the work of P\^{e}cher on near-factorizations) connecting abelian and non-abelian structures, which have recently regained attention.

We present a new explicit framework enabling transfer of difference structures between groups using a semidirect product approach, establishing new tools and constructions, encompassing various previous results and addressing an open problem of Swartz et al.  Motivated by the classic dihedral results, we focus on semidirect products $G \rtimes \mathbb{Z}_2$.  Transfer from abelian to non-abelian groups, and between distinct non-abelian groups, are both possible, and our framework can handle $\lambda$-fold near-factorizations, difference sets and external difference structures with arbitrarily many sets.  We obtain new near-factorizations and difference structures in a range of groups, and can transfer various examples not transferrable by previous approaches. We showcase our approach by establishing a new infinite family of three-set abelian circular external difference families, then producing from this the first infinite family of non-abelian circular external difference families.  
\end{abstract}

\section{Introduction}
Difference sets, internal and external difference families and their variations, are combinatorial structures which have been much-studied.  Difference sets (DSs) and (internal) difference families (DFs) have been studied since the early 20th century \cite{hand}.  External difference families (EDFs), strong external difference families (SEDFs) and circular external difference families (CEDFs) were introduced more recently, motivated by applications in cryptography \cite{HucJefMcC, PatSti, StiVei}; these were  unified and extended via the concept of digraph-defined EDFs \cite{HucJefMcC}.  Each such structure consists of one or more subsets of a group $G$, with the property that the multiset of differences arising within/between the elements of the subsets comprises a constant number $\lambda$ of each non-identity element of $G$. 

These structures have primarily been considered in abelian groups, with many constructions obtained via cyclotomy, group theory and finite geometry, and non-existence results via character theory.  There has been some work on difference sets and families in non-abelian groups, but it is observed in \cite{Swa} that (although the number of nonabelian DSs `` dwarf" the number of abelian DSs), relatively little work has done on developing techniques for the nonabelian setting.  In the external difference family context, an EDF construction arising from a partition of a non-abelian $p$-group is given  in \cite{HucPatWEDF} and a $3$-set dihedral CEDF found by computer search is given in \cite{HucJefMcC}), but the only known non-abelian infinite families are two-set SEDFs in certain dihedral groups (\cite{HucJefMcC, KrePatSti}).

Several of the known theoretical non-abelian results focus on connections between a difference structure in an abelian group and a corresponding difference structure in a non-abelian group. One classic example of this is given in the 1985 paper \cite{Dil}, where Dillon proves the existence of an abelian difference set in any group which has an index $2$ subgroup $G$, if the generalized dihedral extension $Dih(G)$ of $G$ has a difference set  (often called ``Dillon's Dihedral Trick"). Work on difference sets (and their partial and relative variants) by Davis and coauthors in papers such as \cite{Apple,Dav,Swa} explores how the nonabelian case may be approached using ideas from the abelian setting.  In \cite{Swa}, the authors investigate how a ``combinatorial transfer method" may be used to obtain nonabelian difference sets from abelian difference sets. Their work aims to unify various ad hoc approaches in which, starting from a given DS, a combinatorial object such as a design or Cayley graph is built and its automorphism group is exploited to yield new examples of DSs. They give group theoretic conditions for when a difference set in one group $G$ guarantees the existence of a DS in an other group with the same parameters, by considering certain automorphisms of $G$. In the process, they prove a partial converse of Dillon's Dihedral Trick (though all other results of \cite{Swa} are specifically in groups of prime power order).  We note that the proof of Dillon's original trick uses only basic properties of the difference multisets involved, without explicit reference to semidirect product structure.

A factorization of a group $G$ consists of two subsets $A$ and $B$ of $G$ such that $AB=G$. A near-factorization consists of two subsets of a $G$ such that $AB=G \setminus \{g\}$ for some $g \in G$; we may assume that the omitted element $g$ is the identity. There is a correspondence between a near-factorization of a group $G$ and a certain type of two-set external difference family (a generalized strong EDF (GSEDF)) in $G$.  Near-factorizations occur in both abelian and non-abelian groups.  A construction for a near-factorization of the dihedral group $D_{2n}$ is given in \cite{Caen}. In \cite{Pec} P\^{e}cher (motivated by an application in graph theory) gives a method to transfer a near-factorization of $G \times \mathbb{Z}_2$ (where $G$ is an abelian group) to a near-factorization of the generalized dihedral extension of $G$, and this is further discussed in the $D_{2n}$ context by Kreher et al in \cite{KrePatSti} (Kreher et al are the first to highlight the connection with GSEDFs). Near-factorizations were generalized in \cite{KreLiSti} to $\lambda$-fold near-factorizations (where $AB=\lambda(G \setminus \{e\})$), which correspond to two-set GSEDFs.

In this paper, we present a semidirect product framework for transferring difference structures and near-factorizations between different groups.  Our approach is explicit, focusing on the automorphisms involved.  It applies not only to difference sets and $\lambda$-fold near-factorizations, but also to the general setting of digraph-defined external difference families with any number of sets.   Motivated by the results of \cite{KrePatSti,Pec}, we focus on semidirect products of the form $G \rtimes \mathbb{Z}_2$ ($G$ abelian).  These results of P\^{e}cher and Kreher et al feature as a special case of our approach. This paper has strong links with the work of Swartz et al on difference sets, particularly Theorem 1.1 of \cite{Swa}; our approach is more concrete and less group-theoretic, and while we consider a more restricted range of groups than Theorem 1.1, our framework handles a wider range of difference structures (addressing one of their open questions).  We showcase our approach via one particular type of external difference structure: we first establish a new infinite family of abelian three-set CEDFs in $\mathbb{Z}_{(3l^2+1)/2} \times \mathbb{Z}_2$ ($l$ odd) and then apply our transfer method to this, to obtain the first infinite family of non-abelian CEDFs.  Moreover, various examples of near-factorizations and difference structures from the literature, not amenable to transfer via previous methods \cite{KreLiSti, KrePatSti, Pec}, can be transferred to other groups via our tools (in some cases, between different non-abelian groups).

\section{Background}
We begin by recalling some algebraic and combinatorial concepts which will be used in what follows.  

Throughout, $G$ will denote a group. In general, we will use multiplicative notation for arbitrary or non-abelian groups, and additive notation for abelian groups. For a group $G$, we denote by $AUT(G)$ the group of automorphisms of $G$.  For a ring $R$ we denote by $U(R)$ the group of units of $R$.

\subsection{Semidirect products}

For a subset $X \subseteq G \times \mathbb{Z}_2$, let $X_0=\{(x,0): (x,0) \in X\}$ and $X_1=\{(x,1): (x,1) \in X\}$, so that $X$ is the disjoint union of $X_0$ and $X_1$. 
\begin{definition}
Let $G$ and $H$ be groups, and let $\Gamma: H \rightarrow AUT(G)$ be a group homomorphism. For $h \in H$, denote by $\Gamma_h$ the automorphism associated to $h$. We define a new group, called the \textit{semidirect product of $G$ and $H$ with respect to $\Gamma$} (denoted by $G \rtimes_\Gamma H$) as follows:
\begin{itemize}
\item the underlying set is $\{(g,h): g \in G, h \in H\}$, the Cartesian product $G \times H$;
\item the multiplication is given by:
\[(g_1,h_1)*_{\Gamma}(g_2,h_2) = (g_1 \Gamma_{h_1}(g_2),h_1 h_2).\]
\end{itemize}
It can be shown this is a group with identity $(e_G,e_H)$, where the inverse of $(g,h)$ is $(\Gamma_{h^{-1}}(g^{-1}), h^{-1})$.
\end{definition}
We note that, given $G$ and $H$, distinct choices of $\Gamma$ do not necessarily lead to non-isomorphic semidirect products.  For more information on semidirect products, see for example the expository notes \cite{Con} or the textbook \cite{Isa}. 

\begin{example}
When $\Gamma$ is the trivial homomorphism, so $\Gamma_h$ is the identity on $H$ for all $h \in H$, then the semidirect product is in fact the direct product $G \times H$ of $G$ and $H$. 
\end{example}

We recall that for $G=\mathbb{Z}_m$ and $H=\mathbb{Z}_n$, if $m$ and $n$ are coprime then the group $\mathbb{Z}_{mn}$ is isomorphic to the direct product $\mathbb{Z}_m \times \mathbb{Z}_n$ (under the isomorphism $x \mapsto (x \mod m, x \mod n)$).

In this paper, we will focus on direct products of the following type:
\begin{definition}\label{def:GxZ_2}
Let $G$ be a group that has an automorphism $f$ of order $2$.  Consider the homomorphism $\Gamma$ from $\mathbb{Z}_2$ to $AUT(G)$ given by
\[
    \Gamma_k(g) = \begin{cases} 
           f^0(g), & k=0 \\
           f^1(g), & k=1\\
    \end{cases}
    \]
where $f^0(g)=g$ and $f^1(g)=f(g)$.

A semidirect product $G \rtimes_{\Gamma} \mathbb{Z}_2$ may be defined, with multiplication
\[ (g_1,h_1)*_{\Gamma}(g_2,h_2)=(g_1 f^{h_1}(g_2),h_1+h_2).\]
\end{definition}

(An analogous process gives a semidirect product $G \rtimes_{\Gamma} \mathbb{Z}_n$ for any $G$ with an automorphism $f$ such that $f^n \in AUT(G)$ is the identity.  $G$ may be abelian or non-abelian.)

\begin{lemma}\label{lem:nonabelian}
For a group $G$ and an element $f$ of order 2 in $AUT(G)$, the semidirect product $G \rtimes_{\Gamma} \mathbb{Z}_2$ as defined in Definition \ref{def:GxZ_2}  is non-abelian.
\end{lemma}
\begin{proof}
As $f$ has order 2 in $AUT(G)$, there exists $x \in G$ such that $f(x) \neq x$. Consider $(x,0),(x,1) \in G \rtimes_\Gamma \mathbb{Z}_2$. We have:
    \[(x,0) *_\Gamma(x,1)=(xf^0(x),0+1) = (x^2,1)\]
    \[(x,1) *_\Gamma(x,0)=(xf^1(x),1+0) = (xf(x),1)\]
 Here $xf(x) \neq x^2$, so \label{ex:GxZ_2} $(x,0) *_\Gamma(x,1) \neq (x,1) *_\Gamma(x,0)$, and $G \rtimes_{\Gamma} \mathbb{Z}_2$ is non abelian.
\end{proof}
We observe that the ``inversion" mapping behaves differently for different groups: for $\mathbb{Z}_n$ with $ n \geq 3$, the mapping $f\in AUT(\mathbb{Z}_n)$ given by $f(x) = x^{-1}$ is an order-2 automorphism. However, for $f \in AUT(K_4)$ given by $f(x)=x^{-1}$, $f$ is the identity automorphism.
\begin{remark}
Let $f \in AUT(G)$.
\begin{itemize}
\item If $f(x) = x^{-1}$, then $f$ has order $2$ precisely if the group $G$ has at least one element that is not self-inverse (i.e. at least one element of order greater than $2$).
\item If $G=H \times H$ and $f(x,y) = (y,x)$, then $f$ has order $2$ precisely if $G$ has at least one element $(x,y)$ where $y \neq x$.
\end{itemize}
\end{remark}

The next example, which is a special case of Definition \ref{def:GxZ_2}, is used in P\^{e}cher \cite{Pec}, and is referred to by Kreher, Paterson and Stinson in \cite{KrePatSti} as the ``P\^{e}cher transform" (as applied to the case $G=\mathbb{Z}_{n}$ ($n$ odd)).

\begin{example}\label{def:dihedral}
Let $G$ be an abelian group and $H_2=\mathbb{Z}_2$.  Consider the homomorphism $\Phi$ from $\mathbb{Z}_2$ to $AUT(G)$ given by
    \[
    \Phi_x = \begin{cases} 
           y \mapsto y, & z=0 \\
           y \mapsto -y, & z=1\\
    \end{cases}
    \]
The multiplication is:
\[ (g_1,h_1)*_{\Phi}(g_2,h_2)=(g_1+\Phi_{h_1}(g_2),h_1+h_2)\]
which can also be expressed as:
\[ (g_1,h_1)*_{\Phi}(g_2,h_2)=(g_1+(-1)^{h_1}g_2,h_1+h_2)\]
The resulting semidirect product $G \rtimes_\Phi \mathbb{Z}_2$ is known as the generalized dihedral group, and is written $Dih(G)$. 
\end{example}

This semidirect product simplifies to the direct product if all elements of $G$ have order at most $2$.  If $G$ has at least one element of order greater than $2$, then $G \rtimes_{\Phi} \mathbb{Z}_2$ is distinct from the direct product, and by Lemma \ref{lem:nonabelian} it is nonabelian.
The following special case of Example \ref{def:dihedral} is particularly important.
\begin{lemma}\label{lem:dihedraldef}
Let $n>2$. Let $D_{2n}=\langle r,s: r^n=s^2=e, srs=r^{-1} \rangle$ be the dihedral group of order $2n$. Then $\mathbb{Z}_n \rtimes_\Phi \mathbb{Z}_2$ is isomorphic to $D_{2n}$, under the isomorphism $(x,y) \mapsto r^x s^y$.
\end{lemma}

The following well-known result shows that the above constructions account for all groups of order $2p$:
\begin{proposition}
Let $p$ be an odd prime. Every group of order $2p$ is isomorphic either to $\mathbb{Z}_{2p}$ ($\cong \mathbb{Z}_p \times \mathbb{Z}_2$) or to $D_{2p}$ ($\mathbb{Z}_p \rtimes_{\Phi} \mathbb{Z}_2$).
\end{proposition}

What about cases when $\Gamma_1$ in Example \ref{ex:GxZ_2} is not the inversion map?

\begin{example}\label{ex:z8xz2}
Let $G=\mathbb{Z}_8$ and $H=\mathbb{Z}_2$.  Then $AUT(\mathbb{Z}_8)$ has 4 automorphisms, corresponding to multiplication by the elements $1,3,5$ and $7$ of $U(\mathbb{Z}_8)$; all except the identity have order $2$, and it can be checked that $AUT(\mathbb{Z}_8)$ is isomorphic to $\mathbb{Z}_2 \times \mathbb{Z}_2$.  There are 4 homomorphisms from $\mathbb{Z}_2$ to $AUT(\mathbb{Z}_8)$, determined by where they send $1$: $\alpha_1=(x \mapsto x), \beta_1=(x \mapsto 3x), \delta_1=(x \mapsto 5x)$ and $\epsilon_1=(x \mapsto 7x)$. For the corresponding semidirect products, the first yields the direct product.  We see that $\mathbb{Z}_8 \rtimes_{\beta} \mathbb{Z}_2$ has multiplication $(a,b) *_{\beta}(c,d)=(a+3^b c, b+d)$ and $\mathbb{Z}_8 \rtimes_{\delta} \mathbb{Z}_2$ has multiplication $(a,b) *_{\delta}(c,d)=(a+5^b c, b+d)$.  Finally, since $7=-1$ in $\mathbb{Z}_8$, $\mathbb{Z}_8 \rtimes_{\epsilon} \mathbb{Z}_2$ has multiplication $(a,b) *_{\epsilon}(c,d)=(a+(-1)^b c, b+d)$, i.e. this is the dihedral group $D_{16}$ (written additively).  It can be confirmed that these are four non-isomorphic semidirect products.
\end{example}

For $G$ of the form $H\times H$, we can define the following semidirect product with $\mathbb{Z}_2$:

\begin{example}\label{def:wreath}
Let $H$ be an abelian group and let $G=H \times H$. Consider the homomorphism $\Omega$ from $\mathbb{Z}_2$ to $AUT(G)$ given by:
    \[
    \Omega(z) = \begin{cases}
          (x,y) \mapsto (x,y), & z=0 \\
          (x,y) \mapsto (y,x), & z=1. \\
\end{cases}
    \]
A semidirect product may be defined, with multiplication:
\[ ((g_1,g_2),h) *_{\Omega} ((g_1',g_2'),h')=((g_1,g_2)+ \Omega_{h}(g_1',g_2'),h+h')\]
This semidirect product $G \rtimes_{\Omega} \mathbb{Z}_2$ is more commonly known as the \emph {regular wreath product}, denoted $H wr \mathbb{Z}_2$ or $H \wr \mathbb{Z}_2$.
\end{example}

\subsection{Near-factorizations and difference structures}

We introduce notation which we will use throughout the paper. For a multiset $D$ we use $\lambda D$ to denote the multiset consisting of $\lambda$ copies of $D$. 

\begin{definition}\label{def:multi}
Let $K$ be a group with multiplication $*$; let $X$ and $Y$ be subsets of $K$.
\begin{itemize}
\item[(i)]  We denote by $S(X,Y;K)$ the multiset $\{x*y: x \in X, y \in Y\}$. 
\item[(ii)] We denote by $\mathrm{Inv}(X;K)$ the set of inverses of the elements of $X$ in $K$ with respect to the multiplication in $K$.
\item[(iii)] We say that $X$ is \emph{symmetric} in $K$ if $X=\mathrm{Inv}(X;K)$.
\item[(iv)]  We denote by $D(X,Y;K)$ the multiset $S(X,\mathrm{Inv}(Y;K);K)$, and by $D(X; K)$ the multiset $D(X,X;K) \setminus (|X|\{e\})$.
\end{itemize}
\end{definition} 

\begin{remark}\label{rem:symm}
Let $G$ be an abelian group, written additively. For a subset $X$ of the Cartesian product $G \times \mathbb{Z}_2$, we consider the form taken by the symmetric condition in specific semidirect products.
\begin{itemize}
\item In the direct product, $X$ is symmetric precisely if $X=-X$, i.e. if $(x,y) \in X$ if and only if $(-x,y) \in X$.
\item In $Dih(G)$, $X$ is symmetric precisely if $X_0=-X_0$, i.e. if $(x,0) \in X$ if and only if $(-x,0) \in X$.
\item Let $G=H \times H$. In $G \rtimes_{\Omega} \mathbb{Z}_2$, i.e. $H wr \mathbb{Z}_2$, $X$ is symmetric precisely if: $(x,y,0) \in X$ if and only if $(-x,-y,0) \in X$, and $(x,y,1) \in X$ if and only if $(-y,-x,1) \in X$.
\end{itemize}
In particular, any set which is symmetric in the direct product or the wreath product is symmetric in $Dih(G)$, but the converse does not necessarily hold.
\end{remark}

The following definition was given in \cite{Caen, KrePatSti} for the $\lambda=1$ case.  The generalization to $\lambda>1$ was given in \cite{KreLiSti}.

\begin{definition}\label{def:near}
Let $G$ be a group with identity $e$ and let $X,Y$ be subsets of $G$ with $|X|=k$ and $|Y|=l$.  Then $(X;Y)$ form a $\lambda$-fold $(k,l)$-near-factorization of $G$ if the multiset equation
$$ S(X,Y; G)=\lambda(G \setminus \{e\})$$
holds.  If $k=l$ we call this a $\lambda$-fold $k$-near-factorization; if $\lambda=1$ we omit the term $\lambda$-fold. If the parameters are clear from context we simply call it a near-factorization.  
\end{definition}

\begin{example}\label{ex:ABZ_5Z_2}
Consider the sets $A=\{(0,0),(1,1),(4,1)\}$ and $B=\{(0,1),(2,0),(3,0)\}$ in $\mathbb{Z}_5 \times \mathbb{Z}_2$.  They form a $3$-near factorization of the direct product $\mathbb{Z}_5 \times \mathbb{Z}_2$. Since $A=-A$ and $B=-B$, these sets are symmetric in the direct product.
\end{example}

We next present the definitions for various difference structures, both internal and external, which we will work with in this paper.

\begin{definition}
Let $G$ be a group of order $v$ with identity $e$ and let $A$ be a subset of $G$ of size $k$.  Then $A$ is a $(v,k,\lambda)$-difference set (DS) if the following multiset equation holds:
\[ D(A;G)=\lambda(G \setminus \{e\}).\]
\end{definition}

\begin{example}
In $\mathbb{Z}_{11}$, $\{1,3,4,5,9\}$ is an $(11,5,2)$-DS.
\end{example}

\begin{definition}
Let $G$ be a group of order $v$ with identity $e$ and let $m >1$. A family $\{A_0,A_1,\ldots A_{m-1}\}$ of $m$ disjoint subsets of $G$, each of size $l$, is a $(v,m,l,\lambda)$-EDF if the following multiset equation holds:
    \[\bigcup_{i \neq j}D(C_i,C_j;G)=\lambda(G\backslash\{e\}).\]
\end{definition}

\begin{definition}\label{def:GSEDF}
Let $G$ be a group of order $v$ with identity $e$ and let $m >1$. A family $\{A_0,A_1,\ldots A_{m-1}\}$ of $m$ disjoint subsets of $G$ , where $|A_i|=l_i$ is a $(v,m;l_1,\ldots,l_i; \lambda_1, \ldots, \lambda_i)$-GSEDF if the following multiset equation holds for each $i$ ($0 \leq i \leq m-1$):
    \[\bigcup_{j \neq i} D(C_i,C_j;G)=\lambda_i(G\backslash\{e\}).\]
If $|A_0|= \cdots =|A_{m-1}|=l$ (and $\lambda_0= \cdots = \lambda_{m-1}=\lambda$)  then $\{A_0,A_1,\ldots A_{m-1}\}$ is an $(n,m,l; \lambda)$-SEDF. 
\end{definition}

\begin{definition}
    Let $G$ be a group of order $v$ with identity $e$ and let $m >1$. Let $\mathcal{C}=(C_0,C_1,\ldots C_{m-1})$ be an ordered collection of disjoint subsets of $G$, each of size $l$. Then $\mathcal{C}$ is a $(n,m,l,\lambda)$-$c$-CEDF if the following multiset equation holds:
    \[\cup_{i=0}^{m-1}D(C_{i+c \mod m},C_i;G)=\lambda(G\backslash\{0\}).\]
\end{definition}

\begin{example}
In $D_{28}$, let $C_0=\{r^{13},id,r\}$, $C_1=\{r^2,sr^7,r^{12}\}$ and $C_2=\{sr^3,r^7,sr^{11}\}$.  Then $(C_0,C_1,C_2)$ is a $(28,3,3,1)$-$1$-CEDF in $D_{28}$.
\end{example}

The previous three structures can all be considered under the following definition:
\begin{definition}\label{def:H-defined EDF}
Let $G$ be a group of order $v$ with identity $e$, and let $\mathcal{A}=(A_0,\ldots,A_{m-1})$ be an ordered collection of disjoint subsets of $G$, each of size $l$.  Let $H$ be a labelled digraph on $m$ vertices $\{0,1,\ldots,m-1\}$ and let $\E(H)$ be the set of directed edges of $H$.  Then $\mathcal{A}$ is said to be a $(v,m,l,\lambda; H)$-EDF if the following multiset equation holds:
$$ \bigcup_{(i,j) \in \E(H)} D(A_j,A_i;G) = \lambda (G \setminus \{e\}).$$
We will call such a structure a \emph{digraph-defined} EDF.  If we wish to emphasise $H$, we will call it an $H$-defined EDF.
\end{definition}

The motivation for considering near-factorizations in papers such as \cite{KrePatSti} was their connection to two-set generalized strong external difference families. A version of the following result for $\lambda=1$ was given in \cite{KrePatSti}; we provide a proof of the $\lambda$-fold case.

\begin{lemma}\label{NearSEDF}
Let $G$ be a group of order $v$ with identity $e$ and let $A,B \subseteq G$. Then $(A;B)$ is a $\lambda$-fold $(k,l)$-near-factorization of $G$ if and only if $\{A,Inv(B;G)\}$ is a $(v,2;k,l;\lambda,\lambda)$-GSEDF in $G$.
\end{lemma}
\begin{proof}
For a set $X \subseteq G$, we use $X^{-1}$ to denote $\mathrm{Inv}(X;G)$.
Suppose that $\{A,B\}$ is a $(v,2;k,l;\lambda,\lambda)$-GSEDF in $G$: then $D(A,B;G)=S(A,B^{-1};G)=\lambda(G\backslash\{e\})$, hence $(A;B^{-1})$ is a $\lambda$-fold $(k,l)$-near-factorization of $G$.
    Let $(A;B)$ be a $\lambda$-fold $(k,l)$-near factorization of $G$; then $S(A,B;G)=\lambda(G\backslash\{e\})$. It follows that $D(A,B^{-1};G)=S(A,B;G)=\lambda(G\backslash\{e\})$. Now consider $D(B^{-1},A;G) = S(B^{-1},A^{-1};G)$.  We have:
    \[S(B^{-1},A^{-1};G)=\{xy:x \in B^{-1},y \in A^{-1}\}\]
    \[=\{x^{-1}y^{-1}:x \in B,y \in A\}\]
    \[=\{(yx)^{-1}:x \in B,y \in A\}\]
    As $S(A,B;G)=\{yx:x \in B,y \in A\}=\lambda(G\backslash\{e\})$, it follows that $S(B^{-1},A^{-1};G)=\{(yx)^{-1},x \in B,y \in A\}=\lambda(G\backslash\{e\})$. As $e$ does not occur as a difference between $A$ and $B^{-1}$, the sets are disjoint, hence $\{A;B^{-1}\}$ is a $(v,2,;k,l;\lambda,\lambda)$-GSEDF in $G$.
\end{proof}

\begin{example}
    Using Lemma \ref{NearSEDF} and the sets of the $3$-near-factorization of the direct product $\mathbb{Z}_5 \times \mathbb{Z}_2$ given in Example \ref{ex:ABZ_5Z_2}, we can create a SEDF. The sets for the SEDF are $A$ and $-B$; since $B$ is symmetric in $\mathbb{Z}_5 \times \mathbb{Z}_2$ (i.e. $-B=B$), we have that $\{A,B\}$ is a $(10,2;3,1)$-SEDF in $\mathbb{Z}_5 \times \mathbb{Z}_2$.
\end{example}

\section{Sum and difference multisets and near-factorizations}

We begin by establishing a key result about the multiset product of two sets.  We first need the following definition.

\begin{definition}
Let $G$ be an abelian group and let $f \in AUT(G)$.  We define a new map from $G \times \mathbb{Z}_2$ to $G \times \mathbb{Z}_2$ by
\[ (x,y) \mapsto (f(x),y).\]
It is a group homomorphism of the direct product $G \times \mathbb{Z}_2$, since:
$f((x_1,y_1)*(x_2,y_2))=f(x_1 x_2, y_1 y_2)=(f(x_1 x_2),y_1 y_2)=(f(x_1)f(x_2), y_1 y_2)=(f(x_1),y_1)*(f(x_2),y_2)$.  Injectivity and surjectivity follow from the corresponding properties of $f$, and so this new map lies in $AUT(G \times \mathbb{Z}_2)$.  By abuse of notation, we will refer to this map as $f$ when there is no ambiguity.
\end{definition}

\begin{remark}
If $f$ as defined in $AUT(G)$ naturally defines an automorphism on $\mathbb{Z}_2$ (which necessarily acts as the identity map on $\mathbb{Z}_2$), we may view its extension to $G \times \mathbb{Z}_2$ as applying $f$ to pairs $(x,y) \in G \times \mathbb{Z}_2$.  Examples include: when $f$ is the identity, the inversion map or when $G$ has a subgroup $H$ isomorphic to $\mathbb{Z}_2$ and $f \in AUT(G)$ fixes $H$.  
\end{remark}

We are now ready to state our main result of this section.

\begin{theorem}\label{thm:genhom}
Let $G$ be an abelian group. Let $\Gamma, \Theta: \mathbb{Z}_2 \rightarrow AUT(G)$ be group homomorphisms and let $(G \rtimes_{\Gamma} \mathbb{Z}_2,*_{\Gamma})$ and $(G \rtimes_{\Theta} \mathbb{Z}_2,*_{\Theta})$ be the corresponding semidirect products. Denote by $(G \times \mathbb{Z}_2,*)$ the direct product group. Let $X,Y$ be subsets of the Cartesian product $G \times \mathbb{Z}_2$.

If $\Gamma_1(Y)=\Theta_1(Y)$, then:
\begin{itemize}
\item[(i)] $S(X, Y,G \rtimes_{\Gamma} \mathbb{Z}_2)=S(X,Y; G \rtimes_{\Theta} \mathbb{Z}_2)$;
\item[(ii)] $(X;Y)$ is a $\lambda$-fold near-factorization of $(G \rtimes_{\Gamma} \mathbb{Z}_2,*_{\Gamma})$ precisely if $(X;Y)$ is a $\lambda$-fold near-factorization of $(G \rtimes_{\Theta} \mathbb{Z}_2,*_{\Theta})$.
\end{itemize}
\end{theorem}
\begin{proof}
First observe that, since $\Gamma, \Theta$ are group homomorphisms, we have $\Gamma_0=\Theta_0=id$. All unions are multiset unions.
\begin{align*}
& S(X,Y; G \rtimes_{\Gamma} \mathbb{Z}_2)\\ 
&= \bigcup_{(x_1,x_2) \in X} S(\{(x_1,x_2)\},Y; G \rtimes_{\Gamma} \mathbb{Z}_2)\\
&= \bigcup_{(x_1,x_2) \in X} \{(x_1,x_2) *_{\Gamma}(y_1,y_2): (y_1,y_2) \in Y\}\\
&= \bigcup_{(x_1,x_2) \in X} \{(x_1  \Gamma_{x_2}(y_1), x_2 y_2): (y_1,y_2) \in Y\}\\
&= \bigcup_{(x_1,x_2) \in X} \{(x_1,x_2) *(\Gamma_{x_2}(y_1),y_2): (y_1,y_2) \in Y\}\\
&= \bigcup_{(x_1,0) \in X_0} \{(x_1,0) *(\Gamma_{0}(y_1),y_2): (y_1,y_2) \in Y\}
\cup \bigcup_{(x_1,1)\in X_1} \{(x_1,1) *(\Gamma_{1}(y_1),y_2): (y_1,y_2) \in Y\}\\
&= \bigcup_{(x_1,0) \in X_0} \{(x_1,0) * (y_1, y_2): (y_1,y_2) \in Y\} \cup  \bigcup_{(x_1,1)\in X_1} \{(x_1,1) *(\Gamma_{1}(y_1, y_2)): (y_1,y_2) \in Y\} \\
&=  \bigcup_{(x_1,0) \in X_0} \{(x_1,0) * (\Theta_0(y_1, y_2)): (y_1,y_2) \in Y\} \cup  \bigcup_{(x_1,1)\in X_1} \{(x_1,1) *(\Theta_{1}(y_1, y_2)): (y_1,y_2) \in Y\} \\
&=  \bigcup_{(x_1,0) \in X_0} \{(x_1,0) * (\Theta_0(y_1), y_2): (y_1,y_2) \in Y\} \cup  \bigcup_{(x_1,1)\in X_1} \{(x_1,1) *(\Theta_{1}(y_1), y_2): (y_1,y_2) \in Y\} \\
&= \bigcup_{(x_1,x_2) \in X} \{(x_1,x_2) *_{\Theta}(y_1,y_2): (y_1,y_2) \in Y\}\\
&=S(X,Y; G \rtimes_{\Theta} \mathbb{Z}_2)
\end{align*}
and so (i) holds.  Part (ii) follows from part (i) and Definition \ref{def:near}.
\end{proof}

We will often use the following special case:
\begin{corollary}\label{cor:gamma_triv}
Let $G$ be an abelian group, let $\Theta: \mathbb{Z}_2 \rightarrow AUT(G)$ be a group homomorphism and let $(G \rtimes_{\Theta} \mathbb{Z}_2,*_{\Theta})$ be the corresponding semidirect product. Denote by $(G \times \mathbb{Z}_2,*)$ the direct product group. Let $X,Y$ be subsets of the Cartesian product $G \times \mathbb{Z}_2$.

If $Y=\Theta_1(Y)$, then:
\begin{itemize}
\item[(i)] $S(X, Y,G \times \mathbb{Z}_2)=S(X,Y; G \rtimes_{\Theta} \mathbb{Z}_2)$;
\item[(ii)] $(X;Y)$ is a $\lambda$-fold near-factorization of the direct product $G \times \mathbb{Z}_2$ precisely if $(X;Y)$ is a $\lambda$-fold near-factorization of $(G \rtimes_{\Theta} \mathbb{Z}_2,*_{\Theta})$.
\end{itemize}
\end{corollary}

\begin{remark}\label{rem:Y=theta1Y}
We consider the form taken by the condition $Y=\Theta_1(Y)$ in Corollary \ref{cor:gamma_triv} in specific cases.
\begin{itemize}
\item Suppose $\Theta$ is $\Phi$ as given in Example \ref{def:dihedral} (so $\Theta_1$ is the inversion map on $G$, and by extension on $G \times \mathbb{Z}_2$), and $G \rtimes_{\Theta} \mathbb{Z}_2=Dih(G)$.  Then $Y=\Theta_1(Y)$ precisely if $Y=\mathrm{Inv}(Y;G \times \mathbb{Z}_2)$, i.e. precisely if $Y$ is symmetric in $G \times \mathbb{Z}_2$.
\item Suppose $\Theta$ is $\Omega$ as given in Example \ref{def:wreath} (so $G=H \times H$ for some abelian group $H$ and $\Omega_1$ sends $(a,b,c)$ to $(b,a,c)$), and $G \rtimes_{\Theta} \mathbb{Z}_2=H wr\mathbb{Z}_2$. Then $Y=\Theta_1(Y)$ precisely if, for each $(a,b,c) \in Y$, we have $(b,a,c) \in Y$. We will call such a subset $Y$ of $H \times H \times \mathbb{Z}_2$ a $(1,2)$- \emph{flip set}.
\item Suppose $G=\mathbb{Z}_8$ and $\Theta$ is $\delta$ as defined in Example \ref{ex:z8xz2} (so $\Theta_1=x \mapsto 5x$).  Then $Y=\Theta_1(Y)$ precisely if $5Y=Y$.
\end{itemize}
\end{remark}

The first case of Remark \ref{rem:Y=theta1Y} has received attention in the work of Pe\^{c}her \cite{Pec} and Kreher et al \cite{KrePatSti}.  We give an example.
\begin{example}\label{ex:RunNear}
In Corollary \ref{cor:gamma_triv}, let $G=\mathbb{Z}_5$ and let $\Theta$ be $\Phi$ as given in Example \ref{def:dihedral} (so $\Theta_1$ is the inversion map).  Consider the sets $A$ and $B$ in $\mathbb{Z}_5 \times \mathbb{Z}_2$ from Example \ref{ex:ABZ_5Z_2}. Since $A=-A$ and $B=-B$, these sets satisfy the condition of Corollary \ref{cor:gamma_triv}.  As they form a $3$-near factorization of the direct product $G\times \mathbb{Z}_2$, they also form a $3$-near factorization of $\mathbb{Z}_5 \rtimes_\Phi \mathbb{Z}_2=D_{10}$.
\end{example}

When the second semidirect product is $Dih(G)$ in Corollary \ref{cor:gamma_triv}, this leads to the following result which appears in \cite{Pec}.

\begin{corollary}\label{cor:PecGen}
Let $G$ be an abelian group.  Let $X,Y$ be subsets of $G \times \mathbb{Z}_2$. 
If $Y$ is symmetric in the direct product $G \times \mathbb{Z}_2$ then
\begin{itemize}
\item[(i)] $S(X,Y,G \times \mathbb{Z}_2)=S(X,Y, Dih(G))$;
\item[(ii)] if $(X;Y)$ is a $\lambda$-fold near-factorization of $G \times \mathbb{Z}_2$, then $(X;Y)$ is a $\lambda$-fold near-factorization of $Dih(G)$;
\item[(iii)] $Y$ is symmetric in $Dih(G)$.
\end{itemize}
\end{corollary}
\begin{proof}
For (i), apply Corollary \ref{cor:gamma_triv} by taking $\Theta$ to be $\Phi$ as given in Example \ref{def:dihedral}, so $\Theta_1$ is the inversion map.  Since $Y$ is symmetric in $G \times \mathbb{Z}_2$, $Y=\mathrm{Inv}(Y; G \times \mathbb{Z}_2)$, so $Y=\Theta_1(Y)$ in this case and (i) follows. Part (ii) is then an immediate consequence of (i). For (iii), recall that for $(y_1,y_2)\in Y$, its inverse in $Dih(G)$ is:
    \begin{itemize}
        \item $(y_1^{-1},y_2)$, if $y_2=0$, and 
        \item $(y_1,y_2)$, if $y_2=1$.
    \end{itemize}
Since $Y$ is symmetric in $G \times \mathbb{Z}_2$, we have $(y_1^{-1},y_2) \in Y$ and so $\mathrm{Inv}(Y; Dih(G)) \subseteq Y$.  Moreover any $(y_1,0) \in Y$ is the inverse of $(y_1^{-1},0)$ and any $(y_1,1) \in Y$ is self-inverse. Hence $Y=\mathrm{Inv}(Y,Dih(G))$.
\end{proof}

This gives broader context for results obtained in the $\lambda=1$ case by P\^{e}cher \cite{Pec} and later by Kreher, Paterson and Stinson \cite{KrePatSti}.  In \cite{KrePatSti}, the case of symmetric near-factorizations is considered in $\mathbb{Z}_{2n}$ ($n$ odd) (isomorphic to $\mathbb{Z}_n \times \mathbb{Z}_2$), and corresponding near-factorizations are obtained in $D_{2n}$. We note that the direct product results of Corollary \ref{cor:PecGen} hold for $\mathbb{Z}_n \times \mathbb{Z}_2$ for even $n$ also, though the group is not cyclic.

Examples and constructions for $\lambda$-fold near-factorizations with $\lambda>1$ in small abelian groups are presented in \cite{KreLiSti}.  In the tables of results, it is noted whether the near-factorization is equivalent to a symmetric example (and it is observed that in the abelian setting, if one set in a near-factorization is symmetric then the other set is necessarily symmetric also).  

All symmetric $\lambda$-fold near-factorizations in groups of the form $G \times \mathbb{Z}_2$ presented in $\cite{KreLiSti}$ form symmetric $\lambda$-fold near-factorizations of $Dih(G)$ by Corollary \ref{cor:PecGen}.  Similarly to the near-factorizations in \cite{Pec,KrePatSti}, those in $\mathbb{Z}_n \times \mathbb{Z}_2$ lead to near-factorizations in $D_{2n}$.  However results are obtainable in other non-abelian groups, as the following example illustrates.

\begin{example}
A $4$-fold near-factorization for $\mathbb{Z}_4 \times(\mathbb{Z}_2)^2$ is given in \cite{KreLiSti} with one set being
$$ S=\{(0,0,0), (0,0,1),(0,1,0),(1,0,0),(2,1,1),(3,0,0)\}$$
and the other set 
$$T=\{(0,1,1),(1,0,1),(1,1,0),(2,0,0),(2,0,1),(2,1,0),(3,0,1),(3,1,0),(3,1,1)\}.$$
These are both symmetric in $G \times \mathbb{Z}_2$ where $G=\mathbb{Z}_4 \times \mathbb{Z}_2$. So by Corollary \ref{cor:PecGen}, $(S;T)$ is a $4$-fold near-factorization in $Dih(\mathbb{Z}_4 \times \mathbb{Z}_2)$. 
\end{example}

This approach extends beyond the inversion map to many other settings.  We present two examples of near-factorizations in $\mathbb{Z}_8 \times \mathbb{Z}_2$ listed as non-symmetric in \cite{KreLiSti}.  When we consider these in terms of semidirect products which are not dihedral, we can obtain new examples of near-factorizations in different non-abelian groups.

\begin{example}\label{ex:4NFZ8xZ2}
Let $G=\mathbb{Z}_8$ and $H=\mathbb{Z}_2$. Recall from Example \ref{ex:z8xz2} that there are four homomorphisms from $\mathbb{Z}_2$ to $AUT(\mathbb{Z}_8)$, determined by where they send $1$: $\alpha_1=(x \mapsto x), \beta_1=(x \mapsto 3x), \delta_1=(x \mapsto 5x)$ and $\epsilon_1=(x \mapsto 7x)$. All of these behave as the identity on $\mathbb{Z}_2$.

A $4$-fold near-factorization in $\mathbb{Z}_8 \times \mathbb{Z}_2$ is given by $(S;T)$ where:
$$S=\{(0,0),(1,0),(2,0),(5,0),(0,1),(6,1)\}$$ 
and 
$$ T=\{(1,0),(2,0),(4,0),(5,0),(1,1),(3,1),(4,1),(5,1),(6,1),(7,1)\}.$$
Here $S=5S=\delta_1(S)$ and $T=5T=\delta_1(T)$.  So $\alpha_1(T)=\delta_1(T)$ and, by Theorem \ref{thm:genhom}, $(S;T)$ is a $4$-fold near-factorization of the non-abelian group $\mathbb{Z}_8 \rtimes_{\delta} \mathbb{Z}_2$.  Recall that $\mathbb{Z}_8 \rtimes_{\delta} \mathbb{Z}_2$ has multiplication $(a,b) *_{\delta}(c,d)=(a+5^b c, b+d)$.  The multiplication table of $S(S,T; \mathbb{Z}_8 \rtimes_{\delta} \mathbb{Z}_2)$ is shown in Table \ref{table:4NFZ_8betaZ_2}, with the elements in each row in bold which have been permuted with respect to the multiplication table of $S(S,T; \mathbb{Z}_8 \times \mathbb{Z}_2)$.

Note also that $(-1)T=3T$, i.e. $\epsilon_1(T)=\beta_1(T)$ and so we can apply Theorem \ref{thm:genhom} to see that $S(S,T; D_{16})=S(S,T; \mathbb{Z}_8 \rtimes_{\beta} \mathbb{Z}_2)$.  In this case we do not have a near-factorization since (for example) both sets contain the element $(6,1)$ which is self-inverse in both groups, so the product multiset contains the identity.
\end{example}

\begin{table}[h]
\renewcommand{\arraystretch}{1.2}
\centering
\setlength\tabcolsep{3pt}
\begin{tabular}{|c|cccccccccc|} 
\hline
$S \backslash T$ & $(1,0)$ & $(2,0)$ & $(4,0)$ & $(5,0)$ & $(1,1)$ & $(3,1)$ & $(4,1)$ & $(5,1)$ & $(6,1)$ & $(7,1)$  \\ \hline
$(0,0)$ & $(1,0)$ & $(2,0)$ & $(4,0)$ & $(5,0)$ & $(1,1)$ & $(3,1)$ & $(4,1)$ & $(5,1)$ & $(6,1)$ & $(7,1)$ 
\\
$(1,0)$ & $(2,0)$ & $(3,0)$ & $(5,0)$ & $(6,0)$ & $(2,1)$ & $(4,1)$ & $(5,1)$ & $(6,1)$ & $(7,1)$ & $(0,1)$ 
\\
$(2,0)$ & $(3,0)$ & $(4,0)$ & $(6,0)$ & $(7,0)$ & $(3,1)$ & $(5,1)$ & $(6,1)$ & $(7,1)$ & $(0,1)$ & $(1,1)$ 
\\
$(5,0)$ & $(6,0)$ & $(7,0)$ & $(1,0)$ & $(2,0)$ & $(6,1)$ & $(0,1)$ & $(1,1)$ & $(2,1)$ & $(3,1)$ & $(4,1)$ 
\\
$(0,1)$ & $\textbf{(5,1)}$ & $(2,1)$ & $(4,1)$ & $\textbf{(1,1)}$ & $\textbf{(5,0)}$ & $\textbf{(7,0)}$ & $(4,0)$ & $\textbf{(1,0)}$ & $(6,0)$ & $\textbf{(3,0)}$ 
\\
$(6,1)$ & $\textbf{(3,1)}$ & $(0,1)$ & $(2,1)$ & $\textbf{(4,1)}$ & $\textbf{(1,1)}$ & $\textbf{(5,0)}$ & $(2,0)$ & $\textbf{(7,0)}$ & $(4,0)$ & $\textbf{(1,0)}$ 
\\
\hline
\end{tabular}
\caption{Near-factorization in $\mathbb{Z}_8 \rtimes_{\delta} \mathbb{Z}_2$ from Example \ref{ex:4NFZ8xZ2}}
\label{table:4NFZ_8betaZ_2}
\end{table}

\begin{example}\label{ex:3NFZ8xZ2}
A $3$-fold near-factorization in $\mathbb{Z}_8 \times \mathbb{Z}_2$ is given by 
$$S=\{(0,0),(0,1),(1,0),(3,0),(4,0)\}$$
and 
$$T=\{(1,0),(1,1),(2,0),(3,0),(3,1),(4,1),(5,1),(6,0),(7,1)\}.$$  
Since $S=3S$ and $T=3T$, $\alpha_1(T)=\beta_1(T)$, Theorem \ref{thm:genhom} guarantees that $(S;T)$ is a $3$-fold near-factorization in the non-abelian group $\mathbb{Z}_8 \rtimes_{\beta} \mathbb{Z}_2$.  Recall that $\mathbb{Z}_8 \rtimes_{\beta} \mathbb{Z}_2$ has multiplication $(a,b) *_{\beta}(c,d)=(a+3^b c, b+d)$.
\end{example}

In our final application of this section, we illustrate how our transfer method may take a near-factorization from one non-abelian group to another.

\begin{example}\label{ex:order50}
The $1$-fold near-factorization of $Dih(\mathbb{Z}_5 \times \mathbb{Z}_5)$ with sets
\begin{itemize}
\item $X=\{(0,0,0), (0,1,0),(1,0,0),(2,2,0),(0,0,1),(0,1,1),(1,0,1)\}$
\item $Y=\{(1,4,0),(4,1,0),(4,4,0),(0,3,1),(2,2,1),(3,0,1),(3,3,1)\}$.  
\end{itemize}
is given in \cite{Pec}, then reproduced in \cite{KrePatSti} as being equivalent to one of two computationally-found non-equivalent possibilities.  It cannot be converted to a near-factorization of the direct product $(\mathbb{Z}_5)^2 \times \mathbb{Z}_2$ as $Y \neq -Y$.\\
Consider the order-2 automorphisms of $(\mathbb{Z}_5)^2$ given by $\Phi:(x,y) \mapsto (-x,-y)$ and $\Gamma:(x,y) \mapsto (-y,-x)$, and the corresponding semidirect products $(\mathbb{Z}_5)^2 \rtimes_{\Phi} \mathbb{Z}_2$ and $(\mathbb{Z}_5)^2 \rtimes_{\Gamma} \mathbb{Z}_2$, constructed as in Definition \ref{def:GxZ_2}. Clearly the former semidirect product is $Dih(\mathbb{Z}_5 \times \mathbb{Z}_5)$. It can be proved that these two semidirect products are non-isomorphic (the former has trivial centre whereas the latter has a centre of order 5) and that the semidirect product arising from $\Gamma$ is isomorphic to the wreath product $\mathbb{Z}_5 wr \mathbb{Z}_2$.
Observe that 
$$ \Phi_1(Y)=\{(4,1,0),(1,4,0),(1,1,0),(0,2,1),(3,3,1),(2,0,1),(2,2,1)=\Gamma_1(Y).$$
So by Theorem \ref{thm:genhom}, the above near-factorization of $Dih(\mathbb{Z}_5 \times \mathbb{Z}_5)$ yields a $1$-fold near-factorization of the distinct non-abelian group $(\mathbb{Z}_5)^2 \rtimes_{\Gamma} \mathbb{Z}_2$. The corresponding subtraction tables are presented in Table \ref{table:Dih(Z_5^2)A} and Table \ref{table:Dih(Z_5^2)B}.
\end{example}

\begin{table}[h]
\renewcommand{\arraystretch}{1.2}
\centering
\setlength\tabcolsep{3pt}
\begin{tabular}{|c|ccccccc|} 
\hline
$X \backslash Y$ & $(0,0,0)$ & $(0,1,0)$ & $(1,0,0)$ & $(2,2,0)$ & $(0,0,1)$ & $(0,1,1)$ & $(1,0,1)$ \\ \hline
$(1,4,0)$ & $(1,4,0)$ & $(1,0,0)$ & $(2,4,0)$ & $(3,1,0)$ & $(1,4,1)$ & $(1,0,1)$ & $(2,4,1)$ 
\\
$(4,1,0)$ & $(4,1,0)$ & $(4,2,0)$ & $(0,1,0)$ & $(1,3,0)$ & $(4,1,1)$ & $(4,2,1)$ & $(0,1,1)$ 
\\
$(4,4,0)$ & $(4,4,0)$ & $(4,0,0)$ & $(0,4,0)$ & $(1,1,0)$ & $(4,4,1)$ & $(4,0,1)$ & $(0,4,1)$ 
\\
$(0,3,1)$ & $(0,3,1)$ & $(0,2,1)$ & $(4,3,1)$ & $(3,1,1)$ & $(0,3,0)$ & $(0,2,0)$ & $(4,3,0)$ 
\\
$(2,2,1)$ & $(2,2,1)$ & $(2,1,1)$ & $(1,2,1)$ & $(0,0,1)$ & $(2,2,0)$ & $(2,1,0)$ & $(1,2,0)$ 
\\
$(3,0,1)$ & $(3,0,1)$ & $(3,4,1)$ & $(2,0,1)$ & $(1,3,1)$ & $(3,0,0)$ & $(3,4,0)$ & $(2,0,0)$ 
\\
$(3,3,1)$ & $(3,3,1)$ & $(3,2,1)$ & $(2,3,1)$ & $(1,1,1)$ & $(3,3,0)$ & $(3,2,0)$ & $(2,3,0)$ 
\\
\hline
\end{tabular}
\caption{Near-factorization of $Dih(\mathbb{Z}_5 \times \mathbb{Z}_5)$ from Example \ref{ex:order50}}
\label{table:Dih(Z_5^2)A}
\end{table}

\begin{table}[h]
\renewcommand{\arraystretch}{1.2}
\centering
\setlength\tabcolsep{3pt}
\begin{tabular}{|c|ccccccc|} 
\hline
$X \backslash Y$ & $(0,0,0)$ & $(0,1,0)$ & $(1,0,0)$ & $(2,2,0)$ & $(0,0,1)$ & $(0,1,1)$ & $(1,0,1)$ \\ \hline
$(1,4,0)$ & $(1,4,0)$ & $(1,0,0)$ & $(2,4,0)$ & $(3,1,0)$ & $(1,4,1)$ & $(1,0,1)$ & $(2,4,1)$ 
\\
$(4,1,0)$ & $(4,1,0)$ & $(4,2,0)$ & $(0,1,0)$ & $(1,3,0)$ & $(4,1,1)$ & $(4,2,1)$ & $(0,1,1)$ 
\\
$(4,4,0)$ & $(4,4,0)$ & $(4,0,0)$ & $(0,4,0)$ & $(1,1,0)$ & $(4,4,1)$ & $(4,0,1)$ & $(0,4,1)$ 
\\
$(0,3,1)$ & $(0,3,1)$ & $(4,3,1)$ & $(0,2,1)$ & $(3,1,1)$ & $(0,3,0)$ & $(4,3,0)$ & $(0,2,0)$ 
\\
$(2,2,1)$ & $(2,2,1)$ & $(1,2,1)$ & $(2,1,1)$ & $(0,0,1)$ & $(2,2,0)$ & $(1,2,0)$ & $(2,1,0)$ 
\\
$(3,0,1)$ & $(3,0,1)$ & $(2,0,1)$ & $(3,4,1)$ & $(1,3,1)$ & $(3,0,0)$ & $(2,0,0)$ & $(3,4,0)$ 
\\
$(3,3,1)$ & $(3,3,1)$ & $(2,3,1)$ & $(3,2,1)$ & $(1,1,1)$ & $(3,3,0)$ & $(2,3,0)$ & $(3,2,0)$ 
\\
\hline
\end{tabular}
\caption{Near-factorization of $(\mathbb{Z}_5)^2 \rtimes_{\Gamma} \mathbb{Z}_2$ from Example \ref{ex:order50}}
\label{table:Dih(Z_5^2)B}
\end{table}

\subsection{Relating near-factorizations to difference structures}

Our next aim is to utilise the results obtained in the previous section for near-factorizations of semidirect products, to obtain difference structures in non-abelian groups.  This provides new context for known structures which are not yet well-understood (such as non-abelian difference sets) and new examples and constructions of difference structures in groups for which few (or none) are currently known.  In particular, there are no known infinite families of circular external difference families in nonabelian groups.  However, this transition is not a straightforward process.

To convert results from near-factorizations of semidirect products to results on difference structures in these groups, we must take into account that the inverse of a set $X$ in the Cartesian product $G \times \mathbb{Z}_2$ may be different in different semidirect products of $G$ and $\mathbb{Z}_2$.

\begin{example}\label{ex:inverses}
In  $\mathbb{Z}_8 \times \mathbb{Z}_2$, consider the $(16,6,2)$-DS from \cite{KreLiSti}: 
$$S=\{(0,0),(1,0),(2,0),(5,0),(0,1),(6,1)\}.$$
\begin{itemize}
\item In the direct product, the inverse of $(a,b)$ is $(-a,b)$ and so
$$\mathrm{Inv}(S;\mathbb{Z}_8 \times \mathbb{Z}_2)=\{(0,0),(7,0),(6,0),(3,0),(0,1),(2,1)\}.$$
\item In $\mathbb{Z}_8 \rtimes_{\beta} \mathbb{Z}_2$ as defined in Example \ref{ex:z8xz2}, the inverse of $(a,0)$ is $(-a,0)$ and the inverse of $(a,1)$ is $(-3a,1)=(5a,1)$. So
$$\mathrm{Inv}(S,\mathbb{Z}_8 \rtimes_{\delta} \mathbb{Z}_2)=\{(0,0),(7,0),(6,0),(3,0),(0,1),(6,1)\}.$$ 
\item In $\mathbb{Z}_8 \rtimes_{\delta} \mathbb{Z}_2$ as defined in Example \ref{ex:z8xz2}, the inverse of $(a,0)$ is $(-a,0)$ and the inverse of $(a,1)$ is $(-5a,1)=(3a,1)$. So
$$\mathrm{Inv}(S,\mathbb{Z}_8 \rtimes_{\delta} \mathbb{Z}_2)=\{(0,0),(7,0),(6,0),(3,0),(0,1),(2,1)\}.$$ 
\item In $\mathbb{Z}_8 \rtimes_{\epsilon} \mathbb{Z}_2 \cong D_{16}$ as defined in Example \ref{ex:z8xz2}, the inverse of $(a,0)$ is $(-a,0)$ and the inverse of $(a,1)$ is $(-7a,1)=(a,1)$ (i.e. $(a,1)$ is self-inverse). 
So 
$$\mathrm{Inv}(S; D_{16})=\{(0,0),(7,0),(6,0),(3,0),(0,1),(6,1)\}.$$ 
\end{itemize}
So $S$ has two different inverse sets in the four possible semidirect products, and is symmetric in none.
\end{example}

We next present conditions which will allow us to take  difference structures in one semidirect product of $G$ and $\mathbb{Z}_2$, and obtain difference structures in another.

\begin{theorem}\label{thm:StoD}
Let $G$ be an abelian group. Let $\Gamma, \Theta: \mathbb{Z}_2 \rightarrow AUT(G)$ be group homomorphisms and let $(G \rtimes_{\Gamma} \mathbb{Z}_2,*_{\Gamma})$ and $(G \rtimes_{\Theta} \mathbb{Z}_2,*_{\Theta})$ be the corresponding semidirect products. Denote by $(G \times \mathbb{Z}_2,*)$ the direct product group. Let $X,Y$ be subsets of the Cartesian product $G \times \mathbb{Z}_2$.

Suppose that $\mathrm{Inv}(Y;G \rtimes_{\Gamma} \mathbb{Z}_2)=\mathrm{Inv}(Y;G \rtimes_{\Theta} \mathbb{Z}_2)$ (denote this set by $Z$) and that $\Gamma_1(Z)=\Theta_1(Z)$.  Then:  
\begin{itemize}
\item[(i)] $D(X, Y;  G \rtimes_{\Gamma} \mathbb{Z}_2)= D(X,Y; G \rtimes_{\Theta} \mathbb{Z}_2)$.
\item[(ii)]$\{X,Y\}$ is a $(2|G|,2;|X|,|Y|;\lambda,\lambda)$-GSEDF in $(G \rtimes_{\Gamma} \mathbb{Z}_2,*_{\Gamma})$ if and only if $\{X,Y\}$ is a $(2|G|,2;|X|,|Y|;\lambda,\lambda)$-GSEDF in $(G \rtimes_{\Theta} \mathbb{Z}_2,*_{\Theta})$. 
\end{itemize}
If $X=Y$ then
\begin{itemize}
\item[(iii)] $D(X;  G \rtimes_{\Gamma} \mathbb{Z}_2)= D(X; G \rtimes_{\Theta} \mathbb{Z}_2)$.
\item[(iv)] $X$ is a $(2|G|,|X|,\lambda)$-DS in $(G \rtimes_{\Gamma} \mathbb{Z}_2,*_{\Gamma})$ if and only if $X$ is a $(2|G|,|X|,\lambda)$-DS in $(G \rtimes_{\Theta} \mathbb{Z}_2,*_{\Theta})$.
\end{itemize}
\end{theorem}
\begin{proof}
Consider $D(X, Y;  G \rtimes_{\Gamma} \mathbb{Z}_2)$; this is $S(X,\mathrm{Inv}(Y;G \rtimes_{\Gamma} \mathbb{Z}_2);  G \rtimes_{\Gamma} \mathbb{Z}_2)$.  Under the stated assumptions, by Theorem \ref{thm:genhom}, this is $S(X,\mathrm{Inv}(Y;G \rtimes_{\Theta} \mathbb{Z}_2);  G \rtimes_{\Theta} \mathbb{Z}_2)$, i.e. $D(X,Y; G \rtimes_{\Theta} \mathbb{Z}_2)$ as required for (i).  Part (ii) holds by combining (i) with Lemma \ref{NearSEDF}. Part (iii) follows from Definition \ref{def:multi}, since the stated multisets are obtained by removing $|X|$ copies of the identity from the multisets of (i). Part(iv) follows from (iii).
\end{proof}

\begin{example}\label{ex:(16,6,2)delta}
In Theorem \ref{thm:StoD}, take $G=\mathbb{Z}_8$, $\Gamma$ to be the trivial map and $\Theta$ to be $\delta$ as defined in Example \ref{ex:z8xz2} (so $\Theta_1=x \mapsto 5x$).  Consider the $(16,6,2)$-DS in $\mathbb{Z}_8 \times \mathbb{Z}_2$ from Example \ref{ex:inverses}:
$$S=\{(0,0),(1,0),(2,0),(5,0),(0,1),(6,1)\}.$$ 
By Example \ref{ex:inverses}, $\mathrm{Inv}(S;\mathbb{Z}_8 \rtimes_{\delta} \mathbb{Z}_2)=\mathrm{Inv}(S;\mathbb{Z}_8 \times \mathbb{Z}_2)=Z$.  Moreover $Z=5Z$ and so in particular $\Gamma_1(Z)=\Theta_1(Z)$.  Hence $S$ is a $(16,6,2)$-DS in the nonabelian group $\mathbb{Z}_8 \rtimes_{\delta} \mathbb{Z}_2$.  The difference table for this is given in Table \ref{table:(16,6,2)delta}, with the elements in bold which have been permuted with respect to the difference table in the direct product.
\end{example}

\begin{table}[h]
\renewcommand{\arraystretch}{1.2}
\centering
\setlength\tabcolsep{3pt}
\begin{tabular}{|c|cccccc|} 
\hline
$S \backslash Z$ & $(0,0)$ & $(7,0)$ & $(6,0)$ & $(3,0)$ & $(0,1)$ & $(2,1)$   \\ \hline
$(0,0)$ & $-$ & $(7,0)$ & $(6,0)$ & $(3,0)$ & $(0,1)$ & $(2,1)$
\\
$(1,0)$ & $(1,0)$ & $-$ & $(7,0)$ & $(4,0)$ & $(1,1)$ & $(3,1)$
\\
$(2,0)$ & $(2,0)$ & $(1,0)$ & $-$ & $(5,0)$ & $(2,1)$ & $(4,1)$
\\
$(5,0)$ & $(5,0)$ & $(4,0)$ & $(3,0)$ & $-$ & $(5,1)$ & $(7,1)$
\\
$(0,1)$ & $(0,1)$ & $\textbf{(3,1)}$ & $(6,1)$ & $\textbf{(7,1)}$ & $-$ & $(2,0)$ 
\\
$(6,1)$ & $(6,1)$ & $\textbf{(1,1)}$ & $(4,1)$ & $\textbf{(5,1)}$ & $(6,0)$ & $-$ 
\\
\hline
\end{tabular}
\caption{Addition table for difference set in $\mathbb{Z}_8 \rtimes_{\delta} \mathbb{Z}_2$ from Example \ref{ex:(16,6,2)delta}}
\label{table:(16,6,2)delta}
\end{table}

Theorem \ref{thm:StoD}(iii) and (iv) relate to Theorem 1.1 of \cite{Swa}, and some consequences proved in \cite{Swa} for difference sets and partial difference sets (including a partial converse of Dillon's Dihedral Trick).  This provides context for their observation that their partial converse of the Dihedral Trick requires the difference sets to be reversible.   In the $Dih(G)$ setting, we obtain a simpler result (also including a partial converse of Dillon's trick). 

\begin{corollary}\label{cor:Dih_StoD}
Let $G$ be an abelian group.  Let $X,Y$ be subsets of the Cartesian product $G \times \mathbb{Z}_2$. 
Suppose that $Y$ is symmetric in the direct product $G \times \mathbb{Z}_2$.  Then
\begin{itemize}
\item[(i)]$D(X, Y;  G \times \mathbb{Z}_2)= D(X,Y; Dih(G))$.
\item[(ii)] $\{X,Y\}$ is a $(2|G|,2;|X|,|Y|;\lambda,\lambda)$-GSEDF in $G \times \mathbb{Z}_2$ if and only if $\{X,Y\}$ is a $(2|G|,2;|X|,|Y|;\lambda,\lambda)$-GSEDF in $Dih(G)$.
\end{itemize}
If $X=Y$
\begin{itemize}
\item[(iii)] $D(X;  G \times \mathbb{Z}_2)= D(X; Dih(G))$.
\item[(iv)] $X$ is a $(2|G|,|X|,\lambda)$-DS in $G \times \mathbb{Z}_2$ if and only if $X$ is a $(2|G|,|X|,\lambda)$-DS in $Dih(G)$.
\end{itemize}
\end{corollary}
\begin{proof}
 Here $\Gamma$ is the trivial homomorphism and $\Theta$ is $\Phi$ as given in Example \ref{def:dihedral}.  By Corollary \ref{cor:PecGen}, since $Y$ is  symmetric in $G \times \mathbb{Z}_2$, $Y$ is also symmetric in $Dih(G)$.  So $\mathrm{Inv}(Y;G \rtimes_{\Gamma} \mathbb{Z}_2)=Y=\mathrm{Inv}(Y;G \rtimes_{\Theta} \mathbb{Z}_2)$, and so in the notation of Theorem \ref{thm:StoD}, we have that $Z=Y$. Again by the symmetric property, $\Gamma_1(Y)=\Theta_1(Y)$, so the first result holds.  The second result then follows.
\end{proof}

The fact that, in the particular case considered in Corollary \ref{cor:Dih_StoD}, the symmetric condition in the direct product simultaneously guarantees both the necessary near-factorization property and the necessary set-inverse property may explain why taking symmetric sets from $\mathbb{Z}_n \times \mathbb{Z}_2$ ($n$ odd)  to the dihedral group $D_{2n}$, is the only setting so far considered in the literature for external difference families \cite{KrePatSti}. 

The following example illustrates Corollary \ref{cor:Dih_StoD}; we note that the same difference sets appeared as Example 3.5 of \cite{Swa}.
\begin{example}
Consider the following set in the Cartesian product $\mathbb{Z}_4 \times \mathbb{Z}_2 \times \mathbb{Z}_2$:
\[X=\{(0,0,0),(0,0,1),(0,1,0),(1,0,0),(2,1,1),(3,0,0)\}.\]
Taking $G$ to be $\mathbb{Z}_4 \times \mathbb{Z}_2$, $X$ is a subset of $G \times \mathbb{Z}_2$ and $X$ is symmetric in the direct product $G \times \mathbb{Z}_2$.  So by Corollary \ref{cor:Dih_StoD}, $D(X;  G \times \mathbb{Z}_2)= D(X; Dih(G))$.
As noted in \cite{KreLiSti}, $X$ is a $(16,6,2)$ difference set in the direct product $\mathbb{Z}_4 \times \mathbb{Z}_2 \times \mathbb{Z}_2$.  So by Corollary \ref{cor:Dih_StoD}, $X$ is a $(16,6,2)$ difference set in $(\mathbb{Z}_4 \times \mathbb{Z}_2) \rtimes_\Phi \mathbb{Z}_2=Dih(\mathbb{Z}_4 \times \mathbb{Z}_2)$.
\end{example}

While being symmetric in $G \times \mathbb{Z}_2$ guarantees being symmetric in $Dih(G)$, the converse does not necessarily hold (see Remark \ref{rem:symm}).

The following example provides an infinite family of GSEDFs in both cyclic and dihedral groups, of the type described in \cite{KrePatSti}.  

\begin{example}\label{EDFPecExmp}
Consider a $((2k+1)^2+1,2,k,1)$-SEDF with sets:
\begin{itemize}
\item $\{0,1,\ldots,2k-1,2k\}$
\item $\{2k+1,4k+2,\ldots, 4k^2+2k,4k^2+4k+1\}$
\end{itemize}
in $\mathbb{Z}_{4k^2+4k+2} \cong \mathbb{Z}_{2k^2+2k+1}\times \mathbb{Z}_2$.
These arise from a standard SEDF construction given for example in \cite{PatSti}.
These can be expressed as sets in $\mathbb{Z}_{2k^2+2k+1}\times \mathbb{Z}_2$ by taking the labels modulo $2k^2+2k+1$ and $2$ respectively. (The labels will depend on the parity of $k$.  Here we present the case for odd $k$; the even case is similar).  The sets are
\begin{itemize}
\item $\{(0,0),(1,1),\ldots,(2k-1,1),(2k,0)\}$
\item $\{(2k+1,1),(4k+2,0),\ldots, (2k^2+k,1),(k,0),(3k+1,1),\ldots, (2k^2-1,0),(2k^2+2k,1)\}$.
\end{itemize}
Translating all sets by $k$ yields the equivalent two-set SEDF in $\mathbb{Z}_{2k^2+2k+1}\times \mathbb{Z}_2$:
\begin{itemize}
\item $\{(2k^2+k+1,0),(2k^2+k+2,1),\ldots,(k-1,1),(k,0)\}$
\item $\{(k+1,1),(3k+2,0),\ldots, (2k^2,1),(0,0),(2k+1,1),\ldots, (2k^2-k-1,0),(2k^2+k,1)\}$.
\end{itemize}
Both of these sets are symmetric in the direct product, so Corollary \ref{cor:Dih_StoD} applies to give a $(4k^2+4k+2,2,k,1)$-SEDF in $D_{2(2k^2+2k+1)}$.
\end{example}

\section{External difference families with more than $2$ sets}

Up to now, the external difference structures which we have ``transferred" from one semidirect product to another  have all consisted of two sets (arising from their connection with near-factorizations).  To our knowledge, this is the maximum number of sets which have been considered in this context to date.  Now, we present the first such results for external difference families with more than two sets.   This addresses open problem (7) in the Future Work section of \cite{Swa}.

We first set up the necessary results which allow us to transfer our external difference families between different semidirect products.  Since any set in the family may play the role of either $X$ or $Y$ in some $D(X,Y;G)$, we impose conditions on all sets in the family (previously these were only applied to the set $Y$).

\begin{theorem}\label{thm:SD_digraphEDF}
Let $G$ be an abelian group. Let $\Gamma, \Theta: \mathbb{Z}_2 \rightarrow AUT(G)$ be group homomorphisms and let $(G \rtimes_{\Gamma} \mathbb{Z}_2,*_{\Gamma})$ and $(G \rtimes_{\Theta} \mathbb{Z}_2,*_{\Theta})$ be the corresponding semidirect products. Denote by $(G \times \mathbb{Z}_2,*)$ the direct product group. 

Let $A_0,\ldots, A_{m-1}$ be subsets of the Cartesian product $G \times \mathbb{Z}_2$.  Suppose that, for each $0 \leq i \leq m-1$, we have: $\mathrm{Inv}(A_i;G \rtimes_{\Gamma} \mathbb{Z}_2)=\mathrm{Inv}(A_i;G \rtimes_{\Theta} \mathbb{Z}_2)$ (denote this set by $Z_i$) and that $\Gamma_1(Z_i)=\Theta_1(Z_i)$. 

Then $(A_0,A_1,\ldots, A_{m-1})$ is a $(n,m,l,\lambda; H)$-digraph defined EDF in $G\rtimes_{\Gamma} \mathbb{Z}_2$ precisely if it is a $(n,m,l,\lambda; H)$-digraph defined EDF in $G\rtimes_{\Theta} \mathbb{Z}_2$
\end{theorem}
\begin{proof}
By Theorem \ref{thm:StoD}, under the stated conditions,
$D(A_i,A_j,G\rtimes_{\Gamma}\mathbb{Z}_2)=D(A_i,A_j,G\rtimes_{\Theta}\mathbb{Z}_2)$ for all $0 \leq i \neq j \leq m-1$.
Hence:
\[
\bigcup_{(i,j) \in \E(H)}D(A_{j},A_i,G\rtimes_\Phi\mathbb{Z}_2)=\bigcup_{(i,j) \in \E(H)}D(A_{j},A_i,G\rtimes_{\Gamma} \mathbb{Z}_2)
\]
as required.
\end{proof}

Combining Theorem \ref{thm:SD_digraphEDF} with Corollary \ref{cor:Dih_StoD}, we have the following corollary.
\begin{corollary}\label{cor:symmDigDefEDF}
Let $G$ be an abelian group.  Let $A_0,\ldots, A_{m-1}$ be subsets of the Cartesian product $G \times \mathbb{Z}_2$, such that each $A_i$ ($0 \leq i \leq m-1$) is symmetric in the direct product $G \times \mathbb{Z}_2$. 
Then $(A_0,A_1,\ldots, A_{m-1})$ is a $(n,m,l,\lambda; H)$-digraph defined EDF in $G\times \mathbb{Z}_2$ precisely if it is a $(n,m,l,\lambda; H)$-digraph defined EDF in $Dih(G)$.
\end{corollary}

In \cite{KrePatSti}, the following notion of ``strongly symmetric" was introduced for subsets of the dihedral group $D_{2n}$.  We show that this is in fact a rephrasing (in the particular case when $G=\mathbb{Z}_n$) of the condition for a set in the Cartesian product $G \times \mathbb{Z}_2$ to be symmetric in the direct product group $G \times \mathbb{Z}_2$.
\begin{definition}
A subset $T$ of $D_{2n}=\langle r,s:r^n=s^2=e, srs=r^{-1} \rangle$ is \emph{strongly symmetric} if: for all $0 \leq i \leq 1$ and $0 \leq j \leq n-1$, $x=s^i r^j \in T$ if and only if $s^i r^{-j} \in T$.
\end{definition}
This condition may be equivalently expressed as: $r^j \in T \Leftrightarrow r^{-j} \in T$ and $r^j s \in T \Leftrightarrow r^{-j} s \in T$, i.e.  for all $0 \leq x \leq n-1$ and  $0 \leq y \leq 1$, $r^x s^y \in T$ if and only if $r^{-x} s^y \in T$.  \\
For a subset $T$ of the Cartesian product $\mathbb{Z}_n \times \mathbb{Z}_2$, define the subset $T' \subseteq D_{2n}$ by $T'=\{r^x s^y: (x,y) \in T\}$.  Recall from Lemma \ref{lem:dihedraldef} that the map sending $(x,y) \in \mathbb{Z}_n \rtimes_{\Phi} \mathbb{Z}_2$ to $r^x s^y \in D_{2n}$ is a group isomorphism.
\begin{theorem}\label{thm:sym}
Let $X$ be a subset of the Cartesian product $\mathbb{Z}_n \times \mathbb{Z}_2$.  Then $X$ is a symmetric subset of the direct product $\mathbb{Z}_n \times \mathbb{Z}_2$ if and only if $X'$ is a strongly symmetric subset of $D_{2n}$.
\end{theorem}
\begin{proof}
Let $X$ be a symmetric subset of the direct product. For any $r^x s^y \in X'$, $(x,y) \in X$ and hence $(-x,y) \in X$; thus $r^{-x}s^y \in X'$, so $X'$ is strongly symmetric in $D_{2n}$.
Let $X'$ be strongly symmetric in $D_{2n}$, and consider $(x,y) \in X$. As $r^x s^y \in X'$ and $X'$ is strongly symmetric, we have $r^{-x} s^y\in X'$, hence $(-x,y)\in X$, so $X$ is symmetric in $\mathbb{Z}_n \times \mathbb{Z}_2$.
\end{proof}
Combining Corollary \ref{cor:symmDigDefEDF} and Theorem \ref{thm:sym} yields:
\begin{corollary}\label{cor:stronglysymmetricEDF}
$(A_0,A_1,\ldots, A_{m-1})$ is a $(2n,m,l,\lambda; H)$-digraph defined EDF in the direct product $\mathbb{Z}_n \times \mathbb{Z}_2$ with all sets symmetric if and only if $(A_0',A_1',\ldots, A_{m-1}')$ is a $(2n,m,l,\lambda; H)$-digraph defined EDF in $D_{2n}$ with all sets strongly symmetric.
\end{corollary}

 \subsection{New CEDFs in abelian and dihedral groups}

In this section, we construct a new infinite family of CEDFs in non-cyclic abelian groups, and apply our transfer approach to obtain an infinite family of CEDFs in non-abelian groups (we believe this is the first such non-abelian family in the literature).

Constructions for CEDFs have primarily focussed on cyclic groups or the additive groups of finite fields (see for example \cite{BurMerTra, HucJefMcC, PatSti, StiVei}).
In \cite{HucJefMcC} a construction is presented for an $(3l^2+1,3,l,1)$-CEDF in the non-cyclic abelian group $\mathbb{Z}_{\frac{3l^2+1}{2}}\times\mathbb{Z}_2$, where $l\equiv 3 \mod 4$ (we note that the sets given in this construction also yield a CEDF for $l\equiv 1 \mod 4$, although the proof requires some non-trivial adaptation for this case, due to the change in parity of $(l-1)/2$).  Unfortunately, the sets involved are not symmetric in the direct product, nor can they easily be transformed to an equivalent symmetric version, and so Corollary \ref{cor:symmDigDefEDF} cannot be applied to obtain a corresponding CEDF in the dihedral group.

The following result is a new construction for an infinite family of abelian three-set CEDFs in $\mathbb{Z}_{\frac{3l^2+1}{2}}\times\mathbb{Z}_2$, which holds for any odd $l=2k+1$ and has the property that each set is symmetric.  The proof (which does not require any case-split for the $l\equiv 3 \mod 4$ and $l\equiv 1 \mod 4$ cases) follows the same general proof approach as the aforementioned construction in \cite{HucJefMcC}; we provide a self-contained proof for completeness.  Alongside the proof, the reader may find it helpful to consult the subtraction tables which illustrate the $k=2$ case, presented in Example \ref{ex:Z38Z2}.

\begin{theorem}\label{thm:noncyclic}
Let $k \in \mathbb{N}$. Define the following subsets of $\mathbb{Z}_{\frac{3(2k+1)^2+1}{2}}\times\mathbb{Z}_2$:
\begin{itemize}
\item $C_0=\bigcup_{i=0}^{2k}\{(-k+i,0)\}$
\item $C_1=\bigcup_{i=0}^{2k}\{(k+1+(3k+2)i,i)\}$
\item $C_2=\bigcup_{i=0}^{2k}\{(2k+1+(3k+1)i,i+1)\}$
\end{itemize}
where the first component is taken modulo $\frac{3(2k+1)^2+1}{2}$ and the second component is taken modulo $2$.\\
Then $(C_0,C_1,C_2)$ form a $(3(2k+1)^2+1,3,2k+1,1)$-CEDF in the non-cyclic abelian group $\mathbb{Z}_{\frac{3(2k+1)^2+1}{2}}\times\mathbb{Z}_2$ and each $C_i$ ($0 \leq i \leq 2$) is a symmetric set in this group.
\end{theorem}

\begin{proof}
Note $(3(2k+1)^2+1)/2=6k^2+6k+2$ is even and hence $G=\mathbb{Z}_{\frac{3(2k+1)^2+1}{2}}\times\mathbb{Z}_2$ is not a cyclic group.  
Throughout this proof, for simplicity, we will denote $D(X,Y; G)$ by $D(X,Y)$. We will show that $D(C_1,C_0) \cup D(C_2,C_1) \cup D(C_0,C_2)=G \setminus \{0\}$; note that disjointness of the three sets then follows from the fact that $0$ does not occur in the multiset of differences.  Since the difference multiset contains $3(2k+1)^2$ elements by construction, it will suffice to show that each non-identity group element occurs at least once.\\
For $a, b \in \mathbb{Z}_{\frac{3l^2+1}{2}}$ and $x \in \mathbb{Z}_2$, we will use the interval notation: we denote the set $\{(a,x), (a+1,x), \ldots, (b,x)\}$ by $[a,b]\times \{x\}$. \\
First, we determine the difference multisets. 
\begin{align*}
D  (C_1,C_0) &= \bigcup_{i=0}^{2k}\bigcup_{j=0}^{2k}\{(k+1+(3k+2)i-(-k+j),i)\}\\
&= \bigcup_{i=0}^{2k}\bigcup_{j=0}^{2k}\{((3k+2)i+2k+1-j,i)\}\\
&= \bigcup_{i=0}^{2k}[(3k+2)i+1,(3k+2)i+2k+1]\times \{i\}.\\
\end{align*}
We may view these differences as consisting of $2k$ length-$2k$ ``runs" of consecutive elements in the first coordinate, which occur horizontally in the difference table, indexed by $i$. Each run has fixed second coordinate, and the parity of the second coordinate alternates as $i$ takes values from $0$ to $2k$.

Next we have:
\begin{align*}
D(C_2,C_1) &= \bigcup_{i=0}^{2k}\bigcup_{j=0}^{2k}\{(2k+1+(3k+1)i-(k+1+(3k+2)j),i-j+1)\}\\
&= \bigcup_{i=0}^{2k}\bigcup_{j=0}^{2k}\{(3k+1)(i-j+1)-2k-1-j,i-j+1)\}.
\end{align*}
We let $s = i-j+1$: since $i$ and $j$ vary from $0$ to $2k$, $s$ takes values from  $-2k+1$ to $2k+1$.  We change the indices from $i$ and $j$ to $s$ and $j$.  For $-2k+1 \leq s \leq 2k+1$, let  $$J_s=\{j: 0 \leq j \leq 2k \mbox{ and } 0 \leq s+j-1 \leq 2k\}=[0,2k] \cap [1-s,2k+1-s]$$
so that
$$D  (C_2,C_1) = \bigcup_{s=-2k+1}^{2k+1}\bigcup_{j \in J_s} \{((3k+1)s-j-2k-1,s)\}.$$
For $-2k+1 \leq s \leq 0$, we have $J_s=[1-s,2k]$, while for $1 \leq s \leq 2k+1$ we have $J_s=[0,2k+1-s].$  Thus we have:
\[ D  (C_2,C_1) = \bigcup_{s=-2k+1}^{0}
\bigcup_{j = 1-s}^{2k} \{((3k+1)s-j-2k-1,s)\}
\cup\bigcup_{s=1}^{2k+1}\bigcup_{j =0}^{2k+1-s} 
\{((3k+1)s-j-2k-1,s)\}\]
\[= \bigcup_{s=-2k+1}^{0}
[(3k+1)s-4k-1,(3k+2)s-2k-2]\times \{s\}\]\[\cup\bigcup_{s=1}^{2k+1}
[(3k+2)s-4k-2,(3k+1)s-2k-1]\times \{s\}.\]
We may view these differences as ``runs" of consecutive elements (in the first coordinate) occurring diagonally in the difference table The set $J_s$ allows for the adjustment of the length of the differences depending on which diagonal we are considering. Again, we have alternation between $0$ and $1$ in the second coordinate for each run.

Finally:
\begin{align*}
D(C_0,C_2) &=\bigcup_{i=0}^{2k}\bigcup_{j=0}^{2k}\{(-k+i-(2k+1+(3k+1)j),j+1)\}\\
&=\bigcup_{j=0}^{2k}[-(3k+1)j-3k-1,-(3k+1)j-k-1]\times\{j+1\}.
\end{align*}
These differences we may view as ``runs" of consecutive elements (in the first coordinate) occurring vertically in the difference table, indexed by $j+1$, with alternation in the second coordinate. Re-indexing to replace $j$ by $2k-j$, and using the fact that $6k^2+6k+2 \equiv 0 \mod 3(2k + 1)^2 + 1)/2$ and $j \equiv (2k-j) \mod 2$, we have that
\begin{align*}
D(C_0,C_2)=\bigcup_{j=0}^{2k}[(3k+1)j+k+1,(3k+1)(j+1)]\times\{j+1\}.
\end{align*}

Having obtained expressions for each $D(C_{i+1 \mod 3},C_{i})$ ($i \in \{0,1,2\}$), we next check that the difference multiset comprises one occurrence of each non-identity element in $\mathbb{Z}_{\frac{3(2k+1)^2+1}{2}} \times \{0\}$.  We consider the differences corresponding to the following indices, where $ 0 \leq t \leq k-1$.  
\begin{itemize}
\item $i=2t$ in $D(C_1,C_0)$: the difference multiset obtained is
\begin{align*}
&= \bigcup_{t=0}^{k-1}[(3k+2)(2t)+1,(3k+2)(2t)+2k+1]\times \{0\}\\
\end{align*}
\item $s=2(t+1)$ in $D(C_2,C_1)$: the difference multiset obtained is
\begin{flalign*}
&=\bigcup_{t=0}^{k-1}
[(3k+2)(2(t+1))-4k-2,(3k+1)(2(t+1))-2k-1]\times \{0\}\\
&=\bigcup_{t=0}^{k-1} 
[(3k+2)(2t)+2k+2,(3k+1)(2t+1)+k]\times \{0\}\\
\end{flalign*}
\item $j =2t+1$ in $D(C_0,C_2)$: the difference multiset obtained is
\begin{align*}
&=\bigcup_{t=0}^{k-1}[(3k+1)(2t+1)+k+1,(3k+1)(2(t+1))]\times\{0\}
\end{align*}
\item $s=-2t$ in $D(C_2,C_1)$ (here $-2k+1 \leq s \leq 0$):
the difference multiset obtained is
\begin{flalign*}
&= \bigcup_{t=0}^{k-1}
[(3k+1)(-2t)-4k-1,(3k+2)(-2t)-2k-2]\times \{0\}\\
&= \bigcup_{t=0}^{k-1}
[(3k+1)(-2(k-1-t))-4k-1+6k^2+6k+2,\\
&(3k+2)(-2(k-1-t))-2k-2+6k^2+6k+2]\times \{0\}\\
&= \bigcup_{t=0}^{k-1}
[(3k+1)(2(t+1))+1,(3k+2)(2(t+1)))]\times \{0\}\\
\end{flalign*}
Here we have reindexed to replace $t$ by $k-1-t$ and used the fact that $6k^2+6k+2 \equiv 0 \mod \frac{3(2k+1)^2+1}{2}$.
\end{itemize}
The union of these is $[(3k+2)2t+1,(3k+2)2(t+1)]\times\{0\}$.
As $t$ ranges from $ 0 \text{ to } k-1$ we obtain one copy of:
\[[1,6k^2+4k]\times\{0\}\]
Finally, we take $i=2k$ in $D(C_1,C_0)$.  This contributes the set of differences 
$$[6k^2+4k+1,6k^2+6k+1]\times\{0\}.$$

Now we consider the elements of $\mathbb{Z}_{\frac{3(2k+1)^2+1}{2}}\times\{1\}$. Again we consider the differences corresponding to the following indices, where $ 0 \leq t \leq k-1$.  
\begin{itemize}
\item $j=2t$ in $D(C_0,C_2)$: the difference multiset obtained is
\begin{align*}
&=\bigcup_{t=0}^{k-1}[(3k+1)(2t)+k+1,(3k+1)(2t+1)]\times\{1\}
\end{align*}
\item $s=-2(t+1)+1$ in $D(C_2,C_1)$ (here $-2k+1 \leq s \leq 0$): the difference multiset obtained is
\begin{flalign*}
&= \bigcup_{t=0}^{k-1}
[(3k+1)(-2(t+1)+1)-4k-1,(3k+2)(-2(t+1)+1)-2k-2]\times \{1\}\\
&= \bigcup_{t=0}^{k-1}
[(3k+1)(-2(k-1-t+1)+1)-4k-1+6k^2+6k+2,\\
&(3k+2)(-2(k-1-t+1)+1)-2k-2+6k^2+6k+2]\times \{1\}\\
&= \bigcup_{t=0}^{k-1}
[(3k+1)(2t+1)+1,(3k+2)(2t+1)]\times \{1\}\\
\end{flalign*}
Here we have reindexed to replace $t$ by $k-1-t$ and used the fact that $6k^2+6k+2 \equiv 0 \mod \frac{3(2k+1)^2+1}{2}$.
\item $i =2t+1$ in $D(C_1,C_0)$: the difference multiset obtained is
\begin{align*}
&=\bigcup_{t=0}^{k-1}[(3k+2)(2t+1)+1,(3k+2)(2t+1)+2k+1]\times\{1\}
\end{align*}
\item $s=2(t+1)+1$ in $D(C_2,C_1)$: the difference multiset obtained is 
\begin{flalign*}
&=\bigcup_{t=0}^{k-1}
[(3k+2)(2(t+1)+1)-4k-2,(3k+1)(2(t+1)+1)-2k-1]\times \{1\}\\
&=\bigcup_{t=0}^{k-1} 
[(3k+2)(2t+1)+2k+2,(3k+1)(2(t+1))+k]\times \{1\}\\
\end{flalign*}
\end{itemize}
The union of these is $[(3k+2)2t+k+1,(3k+1)2(t+1)+k]\times\{1\}$.
As $t$ ranges from $ 0 \text{ to } k-1$ we obtain one copy of:
\[[k+1,6k^2+3k]\times\{1\}\]

Finally, we take $s=1$ in $D(C_2,C_1)$ and $j=2k$ in $D(C_0,C_2)$.  These contribute the sets of differences: 
$$([6k^2+5k+2,6k^2+6k+1]\cup[0,k])\times\{1\}.$$

and
$$[6k^2+3k+1,6k^2+5k+1]\times\{1\},$$

We have shown that $D(C_1,C_0) \cup D(C_2,C_1) \cup D(C_0,C_2)$ comprises one copy of each non-identity element in $\mathbb{Z}_{\frac{3(2k+1)^2+1}{2}}\times\mathbb{Z}_2$, and so $(C_0,C_1,C_2)$ form a $(3(2k+1)^2+1,3,2k+1,1)$-CEDF in $\mathbb{Z}_{\frac{3(2k+1)^2+1}{2}}\times\mathbb{Z}_2$.

To show that each set is symmetric, we re-index. Letting $j =i-k$, we have:
\begin{itemize}
\item $C_0=\bigcup_{j=-k}^{k}\{(j,0)\}$
\item $C_1=\bigcup_{j=-k}^{k}\{(3k^2+3k+1+(3k+2)j,j+k)\}$
\item $C_2=\bigcup_{j=-k}^{k}\{(3k^2+3k+1+(3k+2)j,j+k+1)\}$
\end{itemize}
Let $C_{0,j}=\{(j,0)\}$, $C_{1,j}=\{(3k^2+3k+1+(3k+2)j,j+k)\}$, $C_{2,j}=\{(3k^2+3k+1+(3k+2)j,j+k+1)\}$.  With this notation, the sets may be re-written as:
\begin{itemize}
\item $C_0= C_{0,0} \cup \bigcup_{j=1}^{k}\{C_{0,j},C_{0,-j}\}$
\item $C_1= C_{1,0} \cup \bigcup_{j=1}^{k}\{C_{1,j},C_{1,-j}\}$
\item $C_2= C_{2,0} \cup \bigcup_{j=1}^{k}\{C_{2,j},C_{2,-j}\}$
\end{itemize}
It can be verified that $C_{i,-j}=-C_{i,j}$ for all $i \in \{0,1,2\}$ and all $0 \leq j \leq k$. We demonstrate this for $i=1$; the other cases follow similarly. 
\begin{align*}
C_{1,-j}&=\{(3k^2+3k+1+(3k+2)(-j),-j+k)\}\\
    &=\{(3k^2+3k+1-(3k+2)j-(6k^2+6k+2),-j-k)\}\\
    &=\{-(3k^2+3k+1+(3k+2)j,j+k)\} \\
    &= -C_{1,j}
\end{align*}
Hence $C_{i}=-C_i$ for $i \in \{0,1,2\}$ and so the sets are symmetric.
\end{proof}

\begin{example}\label{ex:Z38Z2}
Let $k=2$.  Theorem $\ref{thm:noncyclic}$ yields the following sets in $\mathbb{Z}_{38}\times\mathbb{Z}_2$:
\[C_0=\{(0,2),(0,1),(0,0),(0,37),(0,36)\}\]
\[C_1=\{(0,3),(1,11),(0,19),(1,27),(0,35)\}\]
\[C_2=\{(1,5),(0,12),(1,19),(0,26),(1,33)\}.\]
These form a $(76,3,5,1)$-CEDF in $\mathbb{Z}_{38}\times\mathbb{Z}_2$.  The subtraction tables for $D(C_1,C_0)$, $D(C_2,C_1)$ and $D(C_0,C_2)$ are given in Tables \ref{table:C_1C_0}, \ref{table:C_2C_1} and \ref{table:C_0C_2} respectively.
\begin{table}[h]
\renewcommand{\arraystretch}{1.2}
\centering
\setlength\tabcolsep{3pt}
\begin{tabular}{|c|ccccc|} 
\hline
$-$ & $(2,0)$ & $(1,0)$ & $(0,0)$ & $(37,0)$ & $(36,0)$ \\ \hline
$(3,0)$ & $ (1,0)$	& $(2,0)$ & $(3,0)$ & $(4,0)$	& $(5,0)$ \\
$(11,1)$ & $ \textbf{(9,1)}$ & $	\textbf{(10,1)}$ & $	\textbf{(11,1)}	$ & $\textbf{(12,1)}	$ & $\textbf{(13,1)}$ \\
$(19,0)$ & $ (17,0)$	& $(18,0)$ & $(19,0)$ & $(20,0)$	& $(21,0)$ \\
$(27,1)$ & $ \textbf{(25,1)}$ & $	\textbf{(26,1)}$ & $	\textbf{(27,1)}	$ & $\textbf{(28,1)}	$ & $\textbf{(29,1)}$ \\
$(35,0)$ & $ (33,0)$	& $(34,0)$ & $(35,0)$ & $(36,0)$	& $(37,0)$ \\
\hline
\end{tabular}
	
\caption{$D(C_1,C_0)$ for the construction in $\mathbb{Z}_{38} \times \mathbb{Z}_2$ from Example \ref{ex:Z38Z2}}
\label{table:C_1C_0}
\end{table}

\begin{table}[h]
\renewcommand{\arraystretch}{1.2}
\centering
\setlength\tabcolsep{3pt}
\begin{tabular}{|c|ccccc|} 
\hline
$-$ & $(3,0)$ & $(11,1)$ & $(19,0)$ & $(27,0)$ & $(35,0)$ \\ \hline
$(5,1)$ & $ \textbf{(2,1)}$	& $(32,0)$ & $\textbf{(24,1)}$ & $(16,0)$	& $\textbf{(8,1)}$ \\
$(12,0)$ & $ (9,0)$ & $	\textbf{(1,1)}$ & $	(31,0)	$ & $\textbf{(23,1)}	$ & $(15,0)$ \\
$(19,1)$ & $ \textbf{(16,1)}$	& $(8,0)$ & $\textbf{(0,1)}$ & $(30,0)$	& $\textbf{(22,1)}$ \\
$(26,0)$ & $ (23,0)$ & $	\textbf{(15,1)}$ & $ (7,0)	$ & $\textbf{(37,1)}	$ & $(29,0)$ \\
$(33,1)$ & $ \textbf{(30,1)}$	& $(22,0)$ & $\textbf{(14,1)}$ & $(6,0)$	& $\textbf{(36,1)}$ \\
\hline
\end{tabular}
	
\caption{$D(C_2,C_1)$ for the construction in $\mathbb{Z}_{38} \times \mathbb{Z}_2$ from Example \ref{ex:Z38Z2}}
\label{table:C_2C_1}
\end{table}

\begin{table}[h]
\renewcommand{\arraystretch}{1.2}
\centering
\setlength\tabcolsep{3pt}
\begin{tabular}{|c|ccccc|} 
\hline
$-$ & $(5,1)$ & $(12,0)$ & $(19,1)$ & $(26,0)$ & $(33,0)$ \\ \hline
$(2,0)$ & $ \textbf{(35,1)}$	& $(28,0)$ & $\textbf{(21,1)}$ & $(14,0)$	& $\textbf{(7,1)}$ \\
$(1,0)$ & $ \textbf{(34,1)}$ & $(27,0)$ & $	\textbf{(20,1)}	$ & $(13,0)	$ & $\textbf{(6,1)}$ \\
$(0,0)$ & $ \textbf{(33,1)}$	& $(26,0)$ & $\textbf{(19,1)}$ & $(12,0)$	& $\textbf{(5,1)}$ \\
$(37,0)$ & $ \textbf{(32,1)}$ & $(25,0)$ & $	\textbf{(18,1)}	$ & $(11,0)	$ & $\textbf{(4,1)}$ \\
$(36,0)$ & $ \textbf{(31,1)}$	& $(24,0)$ & $\textbf{(17,1)}$ & $(10,0)$	& $\textbf{(3,1)}$ \\
\hline
\end{tabular}
	
\caption{$D(C_0,C_2)$ for the construction in $\mathbb{Z}_{38} \times \mathbb{Z}_2$ from Example \ref{ex:Z38Z2}}
\label{table:C_0C_2}
\end{table}
\end{example}

Finally, applying Corollary \ref {cor:symmDigDefEDF} (specifically, its Corollary \ref{cor:stronglysymmetricEDF} form) to Theorem \ref{thm:noncyclic} yields the first infinite family of non-abelian CEDFs with more than two sets.
\begin{theorem}\label{thm:dihedral}
Let $k \in \mathbb{N}$. Define the following subsets of $D_{3(2k+1)^2+1}$:
\begin{itemize}
\item $C_0=\bigcup_{i=0}^{2k}\{r^{-k+i}\}$
\item $C_1=\bigcup_{i=0}^{2k}\{r^{k+1+(3k+2)i}s^i\}$
\item $C_2=\bigcup_{i=0}^{2k}\{r^{2k+1+(3k+1)i}s^{i+1}\}$
\end{itemize}
where the orders of $r$ and $s$ are $\frac{3(2k+1)^2+1}{2}$ and $2$ respectively.\\
Then $(C_0,C_1,C_2)$ form a $(3(2k+1)^2+1,3,2k+1,1)$-CEDF in the dihedral group $D_{3(2k+1)^2+1}$.
\end{theorem}

\section{Future work}
In this paper, we have presented a new framework to obtain difference structures and near-factorizations in groups where few or no examples were previously known, by using a semidirect product approach.  Inspired by previous work in the area, we have concentrated on the case when $G \rtimes \mathbb{Z}_2$ ($G$ abelian) but we would expect that a similar approach could be applied when $\mathbb{Z}_2$ is replaced by other suitable groups, and this avenue seems worthy of investigation.\\
This paper is the first to show how external difference families with more than two sets can be ``transferred" between groups, and to obtain constructions of this type; it would be of particular interest to see more results in this direction.  A current limitation is that there are very few constructions in the literature for external difference families in groups which are neither cyclic nor the additive groups of finite fields. It would therefore be worthwhile to develop constructions for new infinite families of EDFs in non-cyclic direct products and semidirect products, both for their intrinsic interest and as a first step in pursuing a transfer approach in this setting.

\section{Acknowledgements}
The third author acknowledges the support of The Edwards Summer Research Fellowship in Pure Mathematics at the University of St Andrews during summer 2025, during which part of this work was undertaken.  We thank Martyn Quick for helpful conversations, and Andrea Burgess, Francesca Merola and Tommaso Traetta for pointing out that the construction for the three-set CEDF in \cite{HucJefMcC} is applicable for a wider range of parameters than is stated in that paper.

\end{document}